\documentclass[12pt]{amsart}
\usepackage{amsfonts}
\usepackage{amssymb}
\usepackage{amsmath,amscd}
\input xy
\xyoption {all}

\usepackage{geometry}
\usepackage{amsthm}
\usepackage{amssymb}
\usepackage{latexsym}
\usepackage[dvips]{graphics}

\newtheorem{theorem}{Theorem}[section]

\newtheorem{cor}[theorem]{Corollary}
\newtheorem{defn}[theorem]{Definition}
\newtheorem{example}[theorem]{Example}
\newtheorem{lemma}[theorem]{Lemma}
\newtheorem{prop}[theorem]{Proposition}
\newtheorem{remark}[theorem]{Remark}
\newtheorem*{ack}{Acknowledgements}

\newcommand{\MC}{{\mathcal M}}

\newcommand{\Xs}{{X^{(n)}}}
\newcommand{\Xh}{{X^{[n]}}}

\newcommand{\LB}{\mathbb{L}}
\newcommand{\Z}{\mathbb{Z}}
\newcommand{\N}{\mathbb{N}}
\newcommand{\Q}{\mathbb{Q}}

\newcommand{\C}{\mathbb{C}}
\newcommand{\K}{\mathbb{K}}

\newcommand{\Lfact}[1]{[#1]_{\LB }!}
\newcommand{\Lbinom}[2]{\left[\begin{matrix}{#1}\\{#2}\end{matrix}\right]_{\LB }}

\begin{document}

\title[Characteristic classes of Hilbert schemes]{Characteristic classes of Hilbert schemes of points via symmetric products}

\author[S. Cappell ]{Sylvain Cappell}
\address{S. Cappell: Courant Institute, New York University, 251 Mercer Street, New York, NY 10012, USA}
\email {cappell@cims.nyu.edu}

\author[L. Maxim ]{Laurentiu Maxim}
\address{L. Maxim : Department of Mathematics, University of Wisconsin-Madison,
480 Lincoln Drive,
Madison, WI 53706-1388, USA.}
\email {maxim@math.wisc.edu}

\author[T. Ohmoto ]{Toru Ohmoto}
\address{T. Ohmoto: Department of Mathematics, Hokkaido University, 
Kita 10 Nishi 8, 
Sapporo 060-0810, Japan}
\email {ohmoto@math.sci.hokudai.ac.jp}

\author[J. Sch\"urmann ]{J\"org Sch\"urmann}
\address{J.  Sch\"urmann : Mathematische Institut,
          Universit\"at M\"unster,
          Einsteinstr. 62, 48149 M\"unster,
          Germany.}
\email {jschuerm@math.uni-muenster.de}

\author[S. Yokura ]{Shoji Yokura}
\address{S. Yokura : Department of Mathematics and Computer Science, Faculty of Science, Kagoshima University, 21-35 Korimoto 1-chome, Kagoshima 890-0065, Japan.}
\email {yokura@sci.kagoshima-u.ac.jp}

%\subjclass[2000]{Primary 55S15, 32S35, 32S60, 20C30; Secondary 14C30. }

\date{\today}

%\copyrightinfo{2009}{American Mathematical Society}


\begin{abstract}  We obtain a formula for the generating series of (the push-forward under the Hilbert-Chow morphism of) the Hirzebruch homology characteristic classes of the Hilbert schemes of points for a smooth quasi-projective variety of arbitrary pure dimension. This result is based on a geometric construction of a motivic exponentiation generalizing the notion of motivic power structure, as well as on a formula for the generating series of the Hirzebruch homology characteristic classes of symmetric products. We apply the same methods for the calculation of generating series formulae for the Hirzebruch classes of the push-forwards of ``virtual motives" of Hilbert schemes of a threefold. As corollaries, we obtain counterparts for the MacPherson (and Aluffi) Chern classes of Hilbert schemes of a smooth quasi-projective variety  (resp. for threefolds). For a projective Calabi-Yau threefold, the latter yields a Chern class version of the dimension zero MNOP conjecture.
\end{abstract}

\maketitle

\tableofcontents

\section{Introduction}

Moduli spaces of objects associated with a given space $X$ carry many interesting and  surprising structures. While they reflect back some of the properties of $X$, it is often the case that these moduli spaces carry more geometric %structure 
structures
and bring out seemingly hidden aspects of the geometry and topology of $X$. 
One convincing example of this philosophy is that of the {\it Hilbert schemes of points} on a quasi-projective manifold.  These objects, originally studied in algebraic geometry, are closely related to several branches of mathematics, such as singularities, symplectic geometry, representation theory and  even theoretical physics. \\

The Hilbert scheme $X^{[n]}$ of a quasi-projective manifold $X$ describes collections of $n$ (not necessarily distinct) points on $X$. It is the moduli space %for 
of
zero-dimensional subschemes of $X$ of length $n$.  (Here $X^{[n]}$ already denotes  the reduced scheme structure, which suffices for our applications.) It comes equipped with a natural proper morphism $\pi_n:X^{[n]} \to X^{(n)}$ to the $n$-th symmetric product of $X$, the {\it Hilbert-Chow morphism},  taking a zero-dimensional 
%scheme 
subschemes to its associated zero-cycle. This morphism is  birational for $X$ of dimension at most two, 
%For $X$ of dimension at least $3$ the morphism is not birational 
but otherwise for large $n$ the Hilbert scheme is in general reducible and has components of dimension much larger than that of the symmetric product.\\

The Hilbert schemes of points on a smooth curve $C$ are isomorphic to the corresponding symmetric products of $C$, and they are all smooth. Hilbert schemes of  points on a smooth algebraic surface $S$ are also smooth, and their topology is fairly well understood: there exist generating series formulae for their Betti numbers \cite{G},  Hodge numbers and Hirzebruch genus \cite{GS}, elliptic genus \cite{BL2}, characteristic classes \cite{BNW, NW} etc. Already in this case, there are relations to the enumerative geometry of curves, to moduli spaces of sheaves, to infinite dimensional Lie algebras, to the combinatorics of the symmetric group, and for a $K3$ (or abelian) surface they provide (some of the very few)  examples of compact hyperk\"ahler manifolds.
% (or holomorphically symplectic manifolds).

Hilbert schemes of points on a smooth variety $X$ of dimension $d \geq 3$ are usually not smooth, and much less is known about their properties.
%geometry and topology. 
It is therefore important to find ways to compute  their topological and analytical invariants. \\

In \cite{C} %, 
Cheah finds a generating function which expresses the 
Hodge-Deligne polynomials of Hilbert schemes in terms of the Hodge-Deligne polynomial of $X$ and those of the {\it punctual Hilbert schemes} ${\rm Hilb}^n_{\C^d,0}$ parametrizing zero-dimensional subschemes of length $n$ of $\C^d$ concentrated at the origin. Known properties of the latter  yield (e.g., by using \cite{ES}) explicit formulae (i.e., depending only on the Hodge-Deligne polynomial of $X$) when the Hilbert scheme $X^{[n]}$ is smooth (e.g., if $n \leq 3$ or $d \leq 2$). Cheah's result is refined in \cite{GLM}, where the notion of {\it power structure over a (semi-)ring} is used to express the generating series of classes (in the Grothendieck ring $K_0({var}/{\C})$ of varieties) of Hilbert schemes of points on a quasi-projective manifold of dimension $d$ as an {\it exponent} of that for the affine space $\C^d$. (Recall that by using a power structure over a (semi-)ring $R$, one can make sense of an expression of the form $(1+\sum_{i \geq 1} a_it^i)^m$, for $a_i$ and $m$ in $R$.) The main result of \cite{GLM} is the  motivic identity:
\begin{equation}\label{GLM}
1+\sum_{n \geq 1} [X^{[n]}] \cdot t^n=\left(1+\sum_{n \geq 1} \left[{\rm Hilb}^n_{\C^d,0}\right] \cdot t^n\right)^{[X]} \in K_0(var/\C)[[t]].\end{equation}
Cheah's formula in \cite{C} is obtained from (\ref{GLM}) by 
%an application of 
applying
the pre-lambda ring homomorphism $e(-;u,v)$ defined %by taking 
as
the Hodge-Deligne polynomials, that is,  $$e(-;u,v):K_0(var/\C) \to \Z[u,v];$$
$$e([X];u,v):=\sum_{p,q} \left( \sum_i (-1)^i h^{p,q}(H^i_c(X;\C)) \right) \cdot u^p v^q, $$
with $h^{p,q}$ denoting the Hodge numbers of Deligne's mixed Hodge structure on the cohomology groups $H^i_c(X;\Q)$.

\bigskip

The aim of this paper is to compute a generating series formula for (the push-forward under the Hilbert-Chow morphism of) the motivic Hirzebruch classes ${T_{y}}_*(X^{[n]})$ of Hilbert schemes of points on a  $d$-dimensional quasi-projective manifold. Recall here that the motivic Hirzebruch classes ${T_{y}}_*(Z)$ of a complex algebraic variety $Z$ were defined in \cite{BSY} as an extension to the singular setting of Hirzebruch's cohomology characteristic classes appearing in the generalized Hirzebruch-Riemann-Roch theorem  \cite{H}. The motivic Hirzebruch classes are defined in \cite{BSY} via a {\it motivic Hirzebruch %transformation}
class transformation}
$$T_{y*}:K_0(var/Z) \to H_*(Z):=H^{BM}_{even}(Z) \otimes \Q[y],$$ whose normalization provides a {\it functorial unification} of the Chern class transformation of MacPherson \cite{MP}, Todd class transformation of Baum-Fulton-MacPherson \cite{BFM} and L-class transformation of Cappell-Shaneson \cite{CS1}, respectively, thus answering positively an old question of MacPherson about the existence of such a unifying theory \cite{MP2} (cf. \cite{Yo}). % As to the unification problem I would like to add my Banach Center paper, if I may be allowed to do so.
 Over a point space %, 
 this transformation $T_{y*}(-)$ becomes the $\chi_y$-genus given by the specialization of the Hodge-Deligne polynomial at $(u,v)=(-y,1)$, that is,  
$$\chi_y(-)=e(-;-y,1): K_0(var/\C) \to \Z[y].$$

The main result of this paper %consists (in the above notations) of 
is the following computation of the generating series for the push-forward of Hirzebruch classes of Hilbert schemes in terms of a homological exponentiation which will be explained later on in this introduction:%(in the above notations):
\begin{theorem}\label{thm1} Let $X$ be a smooth complex quasi-projective variety of pure dimension $d$. Denote by $X^{[n]}$ the Hilbert scheme of zero-dimensional subschemes of $X$ of length $n$, and by $\pi_n:X^{[n]} \to X^{(n)}$ the Hilbert-Chow morphism to the $n$-th symmetric product of $X$.  Let ${\rm Hilb}^n_{\C^d,0}$ be the punctual Hilbert scheme of zero-dimension subschemes of length $n$ supported at the origin in $\C^d$. Let $$PH_*(X):=\sum_{n\geq 0}^{\infty} \left( H^{BM}_{even}(X^{(n)}) \otimes \Q[y] \right) \cdot t^n$$ be the {\it Pontrjagin ring}.
Then the following generating series formula for the push-forwards under the Hilbert-Chow morphisms of the un-normalized Hirzebruch classes ${T_{(-y)}}_*(X^{[n]})$ of Hilbert schemes holds in the {\it Pontrjagin ring} $PH_*(X)$:

\begin{eqnarray*}
\sum_{n=0}^{\infty} {\pi_n}_*{T_{(-y)}}_*(X^{[n]}) \cdot t^n &=& \left(1+ \sum_{n=1}^{\infty} \chi_{-y}({\rm Hilb}^n_{\C^d,0}) \cdot t^n \cdot d^n_* \right)^{{T_{(-y)}}_*(X)}\\  &:=&  \left( \prod_{k=1}^{\infty} (1-t^k \cdot d^k_*)^{-\chi_{-y}(\alpha_k)} \right)^{{T_{(-y)}}_*(X)} \\ &:=& \prod_{k=1}^{\infty} (1-t^k \cdot d^k_*)^{-\chi_{-y}(\alpha_k) \cdot {T_{(-y)}}_*(X)},
\end{eqnarray*}
where the $\alpha_k \in K_0(var/{\C})$ are the coefficients appearing in the Euler product for the geometric power structure on the pre-lambda ring $K_0(var/{\C})$, i.e.,
\begin{equation}\label{Euler}
1+\sum_{n \geq 1} \left[{\rm Hilb}^n_{\C^d,0}\right] \cdot t^n=\prod_{k=1}^{\infty} (1-t^k)^{-\alpha_k}.
\end{equation}
\end{theorem}

If $X$ is projective, by identifying  the degrees in the above formula, we recover Cheah's generating series formula for the Hodge polynomials $\chi_{-y}(\Xh)$ of Hilbert schemes.
Theorem \ref{thm1} implies that the classes ${\pi_n}_*{T_{(-y)}}_*(X^{[n]})$ 
in the homology of
the symmetric product $X^{(n)}$ can be calculated in terms of the class ${T_{(-y)}}_*(X)$ of $X$, by using the universal geometric constants $\chi_{-y}({\rm Hilb}^n_{\C^d,0})$ (and respectively, $\chi_{-y}(\alpha_k)$) coming from the punctual
Hilbert schemes, as well as some universal algebraic constants related to the combinatorics of the symmetric groups, codified in the definition of the exponentiation on the right hand side.
Moreover, these parts are completely separated into the base and exponent of the
exponentiation.  Of course, the combinatorics of the symmetric groups only %appears after projecting 
appear after pushing 
down from the Hilbert schemes to the symmetric products, so that we %don't 
do not
get formulas for the classes ${T_{(-y)}}_*(X^{[n]})$ in the homology of the Hilbert schemes
(as done in \cite{BNW, NW} for $X$ a smooth surface).\\
%In the case of $3$-folds, the formula of Theorem \ref{thm1}  is a class version of a recent result \cite{BBS} on the motivic degree-zero {\it Donaldson-Thomas invariants}.

For a surface $X$, it is known by \cite{ES} that $\alpha_k=[\C]^{k-1}$, so in this case all of the 
appearing geometric invariants can be explicitly computed. We therefore obtain the following:
\begin{cor}\label{cormain} If $X$ is a smooth surface, then in the notations of the above theorem we obtain the following closed formula: %Here what is the meaning of "closed"?
\begin{eqnarray}\label{eqa}
\sum_{n=0}^{\infty} {\pi_n}_*{T_{(-y)}}_*(X^{[n]}) \cdot t^n &=& \prod_{k=1}^{\infty} (1-t^k \cdot d^k_*)^{-y^{k-1} \cdot {T_{(-y)}}_*(X)}.
\end{eqnarray}
\end{cor}
While the proof of formula (\ref{eqa}) uses a generalized (motivic) exponentiation (see \S\ref{motexp}), another proof of Corollary \ref{cormain} can be given by using the BBDG decomposition theorem \cite{BBD} % Their paper is cited here, although it might not be necesary since it is so well-known.
along the lines of \cite{GS}.

For the following specializations of the (un-normalized) Hirzebruch classes at the parameters $y=-1, 0, 1$, see also Section \ref{motHir}.
For $y=-1$, formula (\ref{eqa}) specializes to a generating series for the homology $L$-classes $$\widetilde{L}_*(X^{[n]}):=\widetilde{L}^*(TX^{[n]})\cap [X^{[n]}],$$ with $\widetilde{L}^*$ the Atiyah-Singer $L$-class of the tangent bundle (which agrees up to powers of $2$ with the Hirzebruch $L$-class):
\begin{eqnarray}\label{eqb}
\sum_{n=0}^{\infty} {\pi_n}_*{\widetilde{L}}_*(X^{[n]}) \cdot t^n &=& \prod_{k=1}^{\infty} (1-t^k \cdot d^k_*)^{(-1)^k \cdot {\widetilde{L}}_*(X)}.
\end{eqnarray}
Similarly, for $y=0$, formula (\ref{eqa}) yields a generating series for the homology Todd classes $${Td}_*(X^{[n]}):={Td}^*(TX^{[n]})\cap [X^{[n]}],$$ with ${Td}^*$ the Todd class of the tangent bundle:
\begin{eqnarray}\label{eqb2}
\sum_{n=0}^{\infty} {\pi_n}_*{Td_*(X^{[n]}}) \cdot t^n &=& (1-t \cdot d_*)^{-Td_*(X)}=\sum_{n=0}^{\infty} {Td^{BFM}_*(X^{(n)}}) \cdot t^n,
\end{eqnarray}
where the last equality follows from the generating series formula for Hirzebruch classes (or Baum-Fulton-MacPherson Todd classes) of symmetric products \cite{CMSSY,M}, as recalled in Corollary \ref{symHir}. Formula (\ref{eqb2}) fits with the birational invariance of the Baum-Fulton-MacPherson Todd classes for spaces with at most rational singularities, like quotient singularities, as $\pi_n$ is in this case a resolution of singularities.

These specializations of the homology Hirzebruch classes of Hilbert schemes are valid only for {\it smooth} Hilbert schemes, whereas the specialization to $y=1$ in relation to MacPherson Chern classes also holds for singular Hilbert schemes, as explained later on in Corollary \ref{cormain2}. 
\\

%Note that the infinite products appearing in Theorem \ref{thm1} are just a formal de

In general, for a quasi-projective manifold $X$ of arbitrary dimension $d$,  Theorem \ref{thm1} implies that for the computation of the characteristic classes of Hilbert schemes $X^{[n]}$, for $n \leq N$ and $N \in \N$ a fixed integer, we only need to know the exponents $\alpha_1, \cdots, \alpha_N$ in the Euler product decomposition $(\ref{Euler})$, or more precisely only the Hodge polynomials $\chi_{-y}$ of these exponents. Such Euler exponents can in general be computed inductively in terms of the coefficients of the given power series by using Gorsky's inversion formula \cite{Go}[Thm.1].  

As an example, let us assume that $n \leq 3$ (and $d \geq 1$ arbitrary). In this case, the Hilbert scheme $X^{[n]}$ is smooth. Moreover, the Grothendieck class $\left[{\rm Hilb}^n_{\C^d,0}\right]$ of the punctual Hilbert scheme for $n \leq 3$ is given by the formula (e.g., see \cite{BBS}[Rem.3.5], and compare also with \cite{C}[Sect.4]):
\begin{equation}\label{motsm}
\sum_{n=0}^3 \left[{\rm Hilb}^n_{\C^d,0}\right] \cdot t^n=1+t+\Lbinom{d}{1} t^{2}+\Lbinom{d+1}{2}t^{3} \ ,
\end{equation}
where 
\[
\Lbinom{n}{k} := \frac{\Lfact{n}}{\Lfact{n-k}\Lfact{k}} 
\]
and 
\[
\Lfact{n} := (\LB ^{n}-1) (\LB ^{n-1}-1)\dotsb (\LB -1) \ .
\]
(Here we use the notation $\LB:=[\C]$.) This information is sufficient for computing the exponents $\alpha_1$, $\alpha_2$ and $\alpha_3$ in $(\ref{Euler})$ by making use of  the inversion formula of \cite{Go}. More precisely, we obtain in this case that:
\begin{equation}\label{alpha} \alpha_1=1 \ , \ \ \alpha_2= \frac{\LB^d-1}{\LB-1} - 1 \ , 
\ \ \alpha_3=\frac{(\LB^{d+1}-1)(\LB^d-1)}{(\LB^2-1)(\LB-1)} - \frac{\LB^d-1}{\LB-1} \ .\end{equation}
In particular, for $d=1$, we get that $\alpha_2=0$ and $\alpha_3=0$, as one expects since the Hilbert schemes coincide with the symmetric products in this case. Similarly, for $d=2$, formula (\ref{alpha}) reduces to $\alpha_2=\LB$ and $\alpha_3=\LB^2$, which are just special cases of the above-mentioned general formula $\alpha_k=\LB^{k-1}$, which holds for surfaces. Finally, note that $\chi_{-y}(\alpha_2)$ and $\chi_{-y}(\alpha_3)$ vanish for $y=0$, fitting with the birational invariance of the Baum-Fulton-MacPherson Todd classes for spaces with at most rational singularities, like quotient singularities
(as in the surface case).
\\

Let us come back to the general situation and explain the notations used in Theorem \ref{thm1}. First, $$PH_*(X):=\sum_{n=0}^{\infty} \left( H^{BM}_{even}(X^{(n)}) \otimes \Q[y] \right)  \cdot t^n :=\prod_{n=0}^{\infty}  \left( H^{BM}_{even}(X^{(n)}) \otimes \Q[y] \right) $$ is a commutative ring with unit $1 \in H^{BM}_{even}(X^{(0)}) \otimes \Q[y]$ (where $X^{(0)}:=\{pt\}$), with respect to the usual Pontrjagin product $\odot$ induced by 
$$\left( H^{BM}_{even}(X^{(n)}) \otimes \Q[y] \right) \times \left( H^{BM}_{even}(X^{(m)}) \otimes \Q[y] \right) \to H^{BM}_{even}(X^{(n+m)}) \otimes \Q[y].$$
Also,  $d^n:X \to X^{(n)}$ is the composition of the diagonal embedding $X \to X^n$ with the natural projection $X^n \to X^{(n)}$, so that (for $d^1=id_X$) 
$$d^n_*=d_*^{\odot n}: H^{BM}_{even}(X) \otimes \Q[y] \to H^{BM}_{even}(X^{(n)}) \otimes \Q[y].$$
Note that the maps $d^n$ are needed to transport homology classes like  ${T_{(-y)}}_*(X)$
from $X$ to $X^{(n)}$.
Let us denote by $$\Psi_r:H^{BM}_{2k}(-) \otimes \Q[y] \to H^{BM}_{2k}(-) \otimes \Q[y]$$ the {\it $r$-th homological Adams operation} defined by multiplying with $1/{r^k}$ on $H^{BM}_{2k}(-;\Q)$ together with $\Psi_r(y)=y^r$ ($r,k \in \mathbb{N}$). By analogy with the classical formula 
$$(1-t)^{-(\cdot)}:=\lambda_t(\cdot)=\exp\left( \sum_{r=1}^{\infty} \Psi_r (\cdot) \frac{t^r}{r} \right)$$
relating a pre-lambda structure to the Adams operations (see \cite{FL}), we define the group homomorphism
\begin{equation}\label{exp-homology}
(1-t\cdot d_*)^{-(\cdot)}:=\exp \left( \sum_{r=1}^{\infty} \Psi_r d^r_*(\cdot) \frac{t^r}{r} \right) : \left(H^{BM}_{even}(X) \otimes \Q[y],+ \right) \to \left(PH_*(X),\odot \right) .
\end{equation}
%by the formula 
%$$(1-t\cdot d_*)^{-(\cdot)}:=\exp \left( \sum_{r=1}^{\infty} \Psi_r d^r_*(\cdot) \frac{t^r}{r} \right).$$
In analogy with the relation
$$(1-t^k)^{-(\cdot)}=(1-t)^{-(\cdot)}\vert_{t \mapsto t^k}$$ between a pre-lambda structure and the corresponding power structure, 
we define for $k \geq 1$ the group homomorphism
\begin{equation}\label{power-homology}
(1-t^k \cdot d^k_*)^{-(\cdot)}:=P_k \circ \left( (1-t\cdot d_*)^{-(\cdot)}  \right) : \left( H^{BM}_{even}(X) \otimes \Q[y],+\right) \to \left(PH_*(X),\odot\right)\:,
\end{equation}
with $$P_k:PH_*(X) \to PH_*(X)$$ 
the {\it $k$-th power operation}  
on the Pontrjagin ring $PH_*(X)$ defined by the push forwards $p_{k*}^{(n)}$ for the natural maps 
$ p_k^{(n)} :\Xs \to X^{(nk)}$, $n \geq 0$, 
which are induced by the diagonal embeddings $X^n \to (X^n)^k$. Note that $P_k$ is a ring homomorphism with respect to the Pontrjagin product $\odot$, with $P_k\circ P_m=P_{km}$,
$P_1$ the identity and $P_k \circ d^r_*=d^{rk}_*$. Hence, 
$$(1-t^k \cdot d^k_*)^{-(\cdot)}=\exp \left( \sum_{r=1}^{\infty} \Psi_r d^{rk}_*(\cdot) \frac{t^{rk}}{r} \right)=\exp \left( \sum_{r=1}^{\infty}  d^{rk}_*(\Psi_r(\cdot)) \frac{t^{rk}}{r} \right).$$
Finally, the homological exponentiation
\begin{equation}\label{exp-homology2}
(1+\ \sum_{n=1}^{\infty} a_n\cdot t^n \cdot d^n_*)^{-(\cdot)}:=
\prod_{k=1}^{\infty}(1-t^k\cdot d^k_*)^{-(b_k\,\cdot)}
 : \left(H^{BM}_{even}(X) \otimes \Q[y],+ \right) \to \left(PH_*(X),\odot \right) 
\end{equation}
is defined by using the unique Euler product decomposition
$$1+\ \sum_{n=1}^{\infty} a_n\cdot t^n = \prod_{k=1}^{\infty}(1-t^k)^{-b_k}$$
with coefficients $a_n,b_k$ in the pre-lambda ring $H^{BM}_{even}(pt) \otimes \Q[y]=\Q[y]$,
whose Adams operation corresponds to $\Psi_r(a\cdot y^n)=a\cdot y^{nr}$ for $a\in \Q$.

\vspace{1.5mm}

After a suitable re-normalization, and by specializing to $y=1$ (compare with \cite{CMSSY}), our Theorem \ref{thm1} yields a generating series formula for the push-forwards (under the Hilbert-Chow morphisms) of the rationalized MacPherson-Chern classes $c_*(X^{[n]})$ of Hilbert schemes, namely:
\begin{cor} \label{cormain2} Under the notations and hypotheses of Theorem \ref{thm1}, we have the following generating series formula:
\begin{eqnarray*}
\sum_{n=0}^{\infty} {\pi_n}_*c_*(X^{[n]}) \cdot t^n &=& \left(1+ \sum_{n=1}^{\infty} \chi({\rm Hilb}^n_{\C^d,0}) \cdot t^n \cdot d^n_* \right)^{c_*(X)}\\  &:=&  \left( \prod_{k=1}^{\infty} (1-t^k \cdot d^k_*)^{-\chi(\alpha_k)} \right)^{c_*(X)} \\ &:=& \prod_{k=1}^{\infty} (1-t^k \cdot d^k_*)^{-\chi(\alpha_k) \cdot c_*(X)} \in \sum_{n=0}^{\infty}  H^{BM}_{even}(X^{(n)};\Q)  \cdot t^n  .
\end{eqnarray*}
where the operation $(1-t\cdot d_*)^{-(\cdot)}$ is defined here by
\begin{equation}
(1-t\cdot d_*)^{-(\cdot)}:=\exp \left( \sum_{r=1}^{\infty} d^r_*(\cdot) \frac{t^r}{r} \right) .
\end{equation}
\end{cor}

In particular, for a smooth surface $X$, we have $\chi(\alpha_k)=\chi([\C]^{k-1})=1$, so that 
\begin{equation}
\sum_{n=0}^{\infty} {\pi_n}_*c_*(X^{[n]}) \cdot t^n=\prod_{k=1}^{\infty} (1-t^k \cdot d^k_*)^{- c_*(X)}.
\end{equation}
In particular, this recovers Ohmoto's formula \cite{Oh} for the generating series of the {\it orbifold Chern classes}  of symmetric products (used in \cite{BNW}), via the identification 
$${\pi_n}_*c_*(X^{[n]})= c^{orb}_*(X^{(n)})$$ given by the crepant resolution $\pi_n:X^{[n]} \to X^{(n)}$.

Moreover, for a smooth $3$-fold $X$, 
although the coefficients $\alpha_k$ are unknown so far, their Euler characteristics can be deduced from 
Cheah \cite{C} or Behrend-Fantechi \cite{BF}, $\chi(\alpha_k)=k$, i.e.,  a MacMahon type Chern class formula simply arises: 
\begin{equation}\label{200}
\sum_{n=0}^{\infty} {\pi_n}_*c_*(X^{[n]}) \cdot t^n=\prod_{k=1}^{\infty} (1-t^k \cdot d^k_*)^{- k \cdot c_*(X)}.
\end{equation}

%Here we use that fact that in the surface case the Euler coefficients $\alpha_k$ are known, and the Hilbert-Chow morphism is a {\it crepant resolution} of the symmetric product. More precisely, in this case we obtain:
%\begin{equation}
%\sum_{n=0}^{\infty} c_*^{orb}(\Xs) \cdot t^n=\sum_{n=0}^{\infty} {\pi_n}_*c_*(X^{[n]}) \cdot t^n=\prod_{k=1}^{\infty} (1-t^k \cdot d^k_*)^{- c_*(X)}.
%\end{equation}

Our strategy for proving the generating series formula of Theorem \ref{thm1} is based on a nice interplay between our geometric definition of a {\it motivic exponentiation} (generalizing the power structure of \cite{GLM0,GLM}) and a {\it motivic Pontrjagin ring} of the symmetric products, as well as on
our generating series formula \cite{CMSSY} for the motivic Hirzebruch classes of  symmetric products. 

 The same method applies to the calculation of generating series formulae for the Hirzebruch classes of the push-forwards of ``virtual motives" ${\pi_n}_*[X^{[n]}]_{\rm relvir}$ of Hilbert schemes of a threefold $X$ in terms of virtual motives $[{\rm Hilb}^n_{\C^3,0}]_{\rm vir}$ of punctual Hilbert schemes, as introduced and studied in \cite{BBS}. We prove the following result:

\begin{theorem}\label{mainq} For any smooth quasi-projective threefold $X$ the following formula holds:
\begin{equation}
\sum_{n=0}^{\infty} {T_{(-y)}}_*({\pi_n}_*[X^{[n]}]_{\rm relvir}) \cdot t^n = \left(1+ \sum_{n=1}^{\infty} \chi_{-y}([{\rm Hilb}^n_{\C^3,0}]_{\rm vir}) \cdot t^n \cdot d^n_* \right)^{{T_{(-y)}}_*(X)}
\end{equation}
Moreover, 
\begin{equation}\label{111a}
\sum_{n=0}^{\infty} {T_{(-y)}}_*({\pi_n}_*[X^{[n]}]_{\rm relvir}) \cdot (-t)^n =  \prod_{k=1}^{\infty} (1-t^k \cdot d^k_*)^{-\chi_{-y}(\alpha_k) \cdot {T_{(-y)}}_*(X)},
\end{equation}
with coefficients $\alpha_k \in K_0(var/{\C})[\LB^{-1/2}]$ given  by 
$$\alpha_k=\frac{(-\LB^{1/2})^{-k}-(-\LB^{1/2})^{k}}{\LB(1-\LB)}.$$
\end{theorem}
\noindent Here we use the convention $\chi_{-y}(-\LB^{1/2}):=y^{1/2}$ and $\Psi_r(y^{1/2}):=y^{r/2}$, fitting with the convention $\chi(\LB^{1/2}):=-1$ used in \cite{BBS}. After a suitable re-normalization, and by specializing to $y=1$, we get the following virtual counterpart of Corollary \ref{cormain2} for the {\it Aluffi classes} $c_*^A(X^{[n]})$ of the Hilbert schemes, as introduced in \cite{Beh}.

\begin{cor}\label{mainp} For any smooth quasi-projective threefold $X$ the following formula holds:
\begin{equation}\label{333}
\sum_{n=0}^{\infty} {\pi_n}_*(c_*^A(X^{[n]})) \cdot (-t)^n =  \prod_{k=1}^{\infty} (1-t^k \cdot d^k_*)^{-k \cdot {c}_*(X)}.
\end{equation}
\end{cor}
This is a class version of the following degree formula:
\begin{equation}\label{334}
\sum_{n=0}^{\infty}  {\rm deg}(c_0^A(X^{[n]})) \cdot (-t)^n =  \prod_{k=1}^{\infty} (1-t^k)^{-k \cdot \chi(X)}
 =  \left(\prod_{k=1}^{\infty} (1-t^k)^{-k}\right)^{\chi(X)} =:M(t)^{\chi(X)},
\end{equation}
which we obtain for a smooth projective variety $X$ by pushing formula (\ref{333}) down to a point. Here
$M(t)$ denotes the classical MacMahon function. For a projective Calabi-Yau threefold $X$, formula (\ref{334}) is nothing else but the famous {\it dimension zero MNOP conjecture} \cite{MNOP}, because in this case the perfect obstruction theory of the Hilbert scheme $X^{[n]}$ is symmetric, so that the virtual Euler characteristic is exactly the degree of the Aluffi class, see \cite{Beh}:
$$ \chi^{\rm vir}(X^{[n]})={\rm deg}(c_0^A(X^{[n]}).$$
The dimension zero MNOP conjecture was already proved by different groups of authors by different methods, e.g., see \cite{BBS,BF} and the references therein.

\begin{ack} S. Cappell is partially supported by DARPA-25-74200-F6188. L. Maxim is partially supported by NSF-1005338. T. Ohmoto is supported by the JSPS grant No. 21540057.
J. Sch\"urmann is supported by the
SFB 878 ``groups, geometry and actions". S. Yokura is partially supported by Grant-in-Aid for Scientific Research (No. 21540088), the Ministry of Education,
Culture, Sports, Science and Technology (MEXT), and JSPS Core-to-Core Program 18005, Japan.
\end{ack}

%%%%%%%%%%%%%%%%%%%%%%%%%%%%%%%%%%%
%%%%%%%%%%%%%%%%%%%%%%%%%%%%%%%%%%%

\section{Power structures} 

The proof of our main result, Theorem \ref{thm1}, is based on a refinement of the geometric power structure over the Grothendieck (semi-)ring of complex algebraic varieties, which was introduced in \cite{GLM0,GLM}. We recall here the relevant facts about power structures on (semi-)\-rings.

\begin{defn} A power structure over a (semi-)ring $R$ is a map
$$\left( 1+ t R[[t]] \right) \times R \to 1+ t R[[t]] , \ \ \ (A(t), m) \mapsto \left(A(t)\right)^m$$ 
satisfying the following properties:
\begin{enumerate}
\item[(i)] $\left(A(t)\right)^0=1$,
\item[(ii)] $\left(A(t)\right)^1=A(t)$,
\item[(iii)] $\left(A(t) \cdot B(t)\right)^m=\left(A(t)\right)^m \cdot \left(B(t)\right)^m$,
\item[(iv)] $\left(A(t)\right)^{m+n}=\left(A(t)\right)^m \cdot \left(A(t)\right)^n$,
\item[(v)] $\left(A(t)\right)^{mn}=\left( \left(A(t)\right)^n \right)^m$,
\item[(vi)] $(1+t)^m=1+mt+ {\rm higher \ order \ terms}$,
\item[(vii)] $\left(A(t^k)\right)^m=\left(A(t)\right)^m\vert_{t \mapsto t^k}$.
\end{enumerate}
\end{defn}

The geometric definition of a motivic power structure is given by the following result
(\cite{GLM0,GLM}):

\begin{theorem}[Gusein-Zade, Luengo,  Melle-Hern\'andez]\label{geom-power}  
Let  $K_0(var/\C)$ resp. $S_0(var/\C)$ be the Grothendieck (semi-)ring of complex quasi-projective varieties, i.e., the (semi-)\-group generated by the isomorphism classes $[X]$ of such varieties modulo the relation $[X]=[Y]+[X \setminus Y]$ for a Zariski closed subvariety $Y \subset X$, and with the multiplication defined by the cartesian product: $[X_1] \cdot [X_2]:=[X_1 \times X_2]$.
Then for a power series $$A(t)=1+\sum_{i=1}^{\infty} [A_i]t^i \in S_0(var/\C)[[t]]$$ and $[X] \in S_0(var/\C)$, the following expression defines a power structure on $S_0(var/\C)$:
\begin{equation}\label{ps}
\left( A(t) \right)^{[X]}:=1+\sum_{n=1}^{\infty} \left\{  \sum_{\underbar{k} \ :  \sum ik_i=n} \left( (\prod_i X^{k_i}) \setminus \Delta \right) \times \prod_i A_i^{k_i} / \prod_i S_{k_i} \right\} \cdot t^n,
\end{equation}
where $\underbar{k}=\{k_i: i \in \Z_{>0}, k_i \in \Z_{\geq 0} \}$ and $\Delta$ is the large diagonal in $X^{\sum_i k_i}$ consisting of $(\sum_i k_i)$-tuples of points of $X$ with at least two coinciding ones. 
% the phrase"at least two coinciding ones" seems to be a bit unclear, so it should be put in differently
Here %, 
the symmetric group $S_{k_i}$ acts by permuting the corresponding $k_i$ factors in $\prod_i X^{k_i} \supset (\prod_i X^{k_i}) \setminus \Delta$ and the spaces $A_i$ simultaneously.

This power structure on $S_0(var/\C)$ can be uniquely extended to a power structures on $K_0(var/\C)$, as well as on
the localization $\MC:=K_0(var/\C)[\LB^{-1}]$ of  $K_0(var/\C)$ with respect to the affine line $\LB:=[\C]$.
\end{theorem}

\begin{example}\rm\label{ex1}
Let $X^{(n)}:=X^n / S_n$ denote the $n$-th symmetric product of a quasi-projective variety $X$. Then:
\begin{equation}
(1+t+t^2+ \cdots)^{[X]}=1+ \sum_{n=1}^{\infty} \left[{X^{(n)}}\right] \cdot t^n
\end{equation}
\end{example}

\begin{example}\rm\label{ex2}
Let $X^{\{n\}}:=(X^n\setminus \Delta) / S_n$ denote the configuration space of $n$ distinct unlabeled points on a quasi-projective variety $X$ (where $\Delta$ is the large diagonal in $X^n$). Then:
\begin{equation}
(1+t)^{[X]}=1+ \sum_{n=1}^{\infty} \left[{X^{\{n\}}}\right] \cdot t^n
\end{equation}
\end{example}

\begin{defn} A pre-lambda structure on a commutative ring $R$ is a group homomorphism 
$$\lambda_t:(R,+) \to \left( 1+ t R[[t]] , \cdot \right)$$
so that $\lambda_t(m)=1+mt \ ({\rm mod} \  t^2)$. A pre-lambda ring homomorphism is a ring homomorphism between pre-lambda rings which commutes with the pre-lambda structures.
\end{defn}

\begin{example}\rm\label{ex3} The {\it Kapranov zeta function} \cite{K}
$$\lambda_t(X):=1+ \sum_{n=1}^{\infty} \left[{X^{(n)}}\right] \cdot t^n$$
defines a pre-lambda structure on $K_0(var/{\C})$.
\end{example}

\begin{remark}\rm\label{rem1} A pre-lambda structure $\lambda_t(\cdot)=:(1-t)^{-(\cdot)}$ on a ring $R$ determines algebraically a power structure $\left(A(t)\right)^m$ on $R$ via the Euler product decomposition. More precisely, a power series 
$A(t)=1+\sum_{i=1}^{\infty} a_i t^i \in R[[t]]$ admits a unique {\em Euler product decomposition} 
\begin{equation}\label{Euler-prod}
A(t)=\prod_{k=1}^{\infty}(1-t^k)^{-b_k}=\prod_{k=1}^{\infty}\left((1-t)^{-b_k}\vert_{t \mapsto t^k}\right)=:{\rm Exp} \left( \sum_{i \geq 1} b_it^i \right),
\end{equation}
 with $b_k \in R$. In fact, $${\rm Exp}: (tR[[t]],+) \overset{\cong}{\to} (1+tR[[t]], \cdot)$$ 
 defines a group isomorphism.
A power structure on $R$ can now be uniquely defined by using (iii) and (vii) by:
\begin{equation}\label{Euler-power}
\left(A(t)\right)^m:=\prod_{k=1}^{\infty}(1-t^k)^{-m \cdot b_k} .
\end{equation}
\end{remark}

From this point of view, the Kapranov zeta function is just (compare with Example \ref{ex1}):
$$\lambda_t(X)=(1-t)^{-[X]},$$
so that the geometric power structure on $K_0(var/\C)$ of Theorem \ref{geom-power} agrees
with the algebraically defined power structure associated to the pre-lambda structure defined
by the Kapranov zeta function.\\

A ring homomorphism $\phi :R_1 \to R_2$ induces a natural ring homomorphism $\phi:R_1[[t]] \to R_2[[t]]$ defined by $\phi(\sum_i a_it^i):=\sum_i \phi(a_i)t^i$. Then Remark \ref{rem1} yields the following:
\begin{prop}\label{prop1} A pre-lambda ring homomorphism $\phi:R_1 \to R_2$ respects the corresponding power structures, i.e., 
\begin{equation} \phi\left( A(t)^m \right) = \left( \phi(A(t)) \right)^{\phi(m)}.
\end{equation}
\end{prop}

As an application of power structures %, 
the authors of \cite{GLM} prove the following result:
\begin{theorem}[Gusein-Zade, Luengo,  Melle-Hern\'andez] \label{Hilbert-GLM}
For a smooth quasi-projective variety $X$ of pure dimension $d$ %, 
the following identity holds in $K_0(var/{\C})[[t]]$:
\begin{equation}\label{eqmain}
1+\sum_{n \geq 1} [X^{[n]}] \cdot t^n=\left(1+\sum_{n \geq 1} \left[{\rm Hilb}^n_{\C^d,0}\right] \cdot t^n\right)^{[X]},
\end{equation}
where, as before, $X^{[n]}={\rm Hilb}^n_X$ is the Hilbert scheme of $n$ points on $X$, and ${\rm Hilb}^n_{\C^d,0}$ denotes the punctual Hilbert scheme of zero-dimensional subschemes of length $n$ supported at the origin of $\C^d$.
\end{theorem}
By Proposition \ref{prop1} %, 
one can derive specializations of formula (\ref{eqmain}) by applying various homomorphisms of pre-lambda rings. For example, 
Cheah's formula \cite{C} is obtained from (\ref{eqmain}) by %an application of 
applying the pre-lambda ring homomorphism defined %by taking 
as 
the Hodge-Deligne polynomials, that is,  
$$e(-;u,v):K_0(var/\C) \to \Z[u,v]$$ 
$$e([X];u,v):=\sum_{p,q} \left( \sum_i (-1)^i h^{p,q}(H^i_c(X;\C)) \right) \cdot u^p v^q.$$
Note that a special case of this is the $\chi_{y}$-genus
$$\chi_{y}=e(-;-y,1): K_0(var/\C) \to \Z[y]$$ defined only in terms of the Hodge filtration by 
$$\chi_{y}([X]):=\sum_{i,p} (-1)^i {\rm dim}_{\C} Gr^p_F H^i_c(X;\C) \cdot (-y)^p.$$
Here  
the pre-lambda structure on the polynomial ring $\Z[u_1,\cdots,u_r]$ in $r$ variables ($r \geq 1$) is defined by: 
\begin{equation}
\lambda_t \left(\sum_{\vec{k}\in \Z^r_{\geq 0}} a_{\vec{k}}\cdot \vec{u}^{\;\vec{k}}\right)
:=\prod_{\vec{k}\in \Z^r_{\geq 0}} \left(1-\vec{u}^{\;\vec{k}}\cdot t \right)^{-a_{\vec{k}}} \:,
\end{equation}
with $\vec{k}:=(k_1,\cdots,k_r) \in \Z^r_{\geq 0}$, $a_{\vec{k}} \in \Z$, and $\vec{u}^{\;\vec{k}}:=u_1^{k_1}\cdots u_r^{k_r}$, so that the corresponding Adams operation 
$\Psi_r$ for $r\in \N$ is given by $$\Psi_r\left( a_{\vec{k}}\cdot \vec{u}^{\;\vec{k}}\right):= a_{\vec{k}}\cdot \vec{u}^{\;r\cdot \vec{k}}\:.$$

\bigskip

By using the Euler product decomposition for the geometric power structure on the pre-lambda ring $K_0({var}/{\C})$, i.e.,
\begin{equation}
1+\sum_{n \geq 1} \left[{\rm Hilb}^n_{\C^d,0}\right] \cdot t^n=\prod_{k=1}^{\infty} (1-t^k)^{-\alpha_k}, 
\end{equation}
we can %re-write 
rewrite formula (\ref{eqmain}) in term of the exponents $\alpha_k \in K_0({var}/{\C})$ as
\begin{equation}\label{eqmain2}
1+\sum_{n \geq 1} [X^{[n]}] \cdot t^n=\left( \prod_{k=1}^{\infty} (1-t^k)^{-\alpha_k}\right)^{[X]}.\end{equation} 

For a smooth surface $X$ %, 
it is known that $\alpha_k=\LB^{k-1}$, where, as above, $\LB:=[\C]$. So, (\ref{eqmain}) becomes in this case (cf. \cite{GLM0})
the following:
\begin{equation}\label{eqmain3}
1+\sum_{n \geq 1} [X^{[n]}] \cdot t^n=\prod_{k \geq 1} (1-t^k)^{-\LB^{k-1}[X]}.
\end{equation}

%%%%%%%%%%%%%%%%%%%%%%%%%%%%%%%%%%%%%
%%%%%%%%%%%%%%%%%%%%%%%%%%%%%%%%%%%%%

\section{(Motivic) Pontrjagin (semi-)rings}

\subsection{Relative motivic Grothendieck (semi-)group} Let $K_0(var/X)$ be the relative motivic Grothendieck group of algebraic varieties over $X$, as introduced by Looijenga \cite{Lo} in relation to motivic integration. $K_0(var/X)$ is the quotient of the free abelian group of isomorphism classes of algebraic morphisms $Y \to X$ by the ``scissor" relation:
$$[Y \to X]=[Z \to Y \to X] + [Y \setminus Z \to Y \to X]$$
for $Z \subset Y$ a closed algebraic subvariety of $Y$. If we let  $Z = Y_{red}$ we deduce that these classes $[Y \to X]$ depend only on the underlying reduced spaces. By resolution of singularities, $K_0(var/X)$ is generated by classes $[Y \to X]$ with $Y$ smooth, pure
dimensional, and proper over $X$. Of course, if $X$ is a point space, we get back the motivic Grothendieck group $K_0(var/\C)$ discussed earlier.

For any morphism $f:X' \to X$ we have a functorial push-forward 
$$f_!:K_0(var/X') \to K_0(var/X) \ , \ [Z \overset{h}{\to} X'] \mapsto [Z \overset{f \circ h}{\to} X].$$
Moreover, an external product
$$ \boxtimes : K_0(var/X) \times K_0(var/X') \to K_0(var/X\times X')$$ is defined by the formula:
$$[Z \to X] \boxtimes [Z' \to X'] =[Z \times Z' \to X \times X'].$$
Similar results apply to the corresponding relative Grothendieck semi-groups $S_0(var/X)$
as studied in \cite{GLM0,GLM}.

\subsection{Pontrjagin rings and operations}
Let $F$ be a functor to the category of abelian (semi-)groups with unit $0$ defined on complex quasi-projective varieties, covariantly functorial for all (proper) morphisms. Assume $F$ is also endowed with a commutative, associative and bilinear cross-product $\boxtimes$ commuting with (proper) push-forwards $(-)_*$, with a unit $1\in F(pt)$. Our main examples for $F(X)$ are:  the relative motivic Grothendieck (semi-)group $K_0(var/X)$ resp. $S_0(var/X)$, or suitable localizations of it like $\MC$,  the Borel-Moore homology  $H_*(X):=H^{BM}_{even}(X) \otimes R$ with $R=\Q,\Q[y]$, or the group $CF(X)$ of (algebraically) constructible functions on $X$.

\begin{defn} For a fixed complex quasi-projective variety $X$ %, 
we define the commutative Pontrjagin (semi-)ring $\left(PF(X), \odot \right)$ by
$$PF(X):=\sum_{n=0}^{\infty}  F(\Xs) \cdot t^n 
:=\prod_{n=0}^{\infty}  F(\Xs)
,$$
with product $\odot$ induced via
$$\odot:F(\Xs) \times F(X^{(m)}) \overset{\boxtimes}{\to} F(\Xs \times X^{(m)}) \overset{(-)_*}{\to} F(X^{(n+m)}),$$
and unit $1 \in F(X^{(0)})=F(pt)$.
\end{defn}

It is easy to see that, if $f:X \to Y$ is a (proper) morphism, then we get an induced (semi-)ring homomorphism $$f_*:=(f^{(n)}_*)_n : PF(X) \to PF(Y),$$ with $f^{(n)}:\Xs \to Y^{(n)}$ the corresponding (proper) morphism on the $n$-th symmetric products.

\begin{defn} The $k$-th power operation $P_k:PF(X) \to PF(X)$ for $k\geq 1$ is the (semi-)ring homomorphism $$P_k:=\left({p_{k*}^{(n)}}:F(\Xs) \to F(X^{(nk)}) \right)_n$$ defined 
by the push forwards $p_{k*}^{(n)}$ for the natural maps 
$p_k^{(n)} :\Xs \to X^{(nk)}$ induced
by the diagonal embeddings $X^n \to (X^n)^k \cong X^{nk}$, $n \geq 0$, with  $P_k\circ P_m=P_{km}$
and $P_1$ the identity. 
\end{defn}

\subsection{Motivic exponentiation}\label{motexp}
Given a quasi-projective variety $X$, we extend the notion of power structure $\left( A(t) \right)^{[X]} \in K_0(var/{\C})[[t]]$ from \cite{GLM0,GLM} to an operation $\left( A(t) \right)^X$ associating to a normalized power series $A(t)=1+\sum_i[A_i]t^i \in K_0(var/{\C})[[t]]$ an element 
$$\left( A(t) \right)^X \in \sum_{n\geq 0}  K_0(var/\Xs) \cdot t^n =:PK_0(var/X)$$ in the Pontrjagin ring $PK_0(var/X)$ of the symmetric products of $X$ associated to the relative motivic Grothendieck groups (and similarly for the corresponding Grothendieck semi-groups). This extension is based on the geometric formula (\ref{ps}) 
$$\left( A(t) \right)^{[X]}:=1+\sum_{n=1}^{\infty} \left\{  \sum_{\underbar{k} \ :  \sum ik_i=n} \left( (\prod_i X^{k_i}) \setminus \Delta \right) \times \prod_i A_i^{k_i} / \prod_i S_{k_i} \right\} \cdot t^n$$
for the power structure of \cite{GLM0,GLM} on the semi-ring $S_0(var/{\C})$, as it will be explained below.\\

The $n$-th symmetric product  $\Xs:=X^n/S_n$ of $X$ parametrizes effective zero-cycles of degree $n$ on $X$, i.e., formal linear combinations $\sum_{i=1}^l n_i [x_i]$ of points $x_i$ in $X$ with non-negative integer coefficients $n_i$, so that $\sum_{i=1}^l n_i=n$. 
$\Xs$ has  a natural stratification into locally closed subschemes defined in terms of the partitions of $n$. More precisely, to any partition $\nu:=(n_1, \cdots, n_l)$ of $n$ one associates a sequence $\underbar{k}:=(k_1,\cdots, k_n)$, with $k_i$ denoting the number of times $i$ appears among the $n_j$'s. The length of such a partition is defined by $l(\nu):=l=\sum_i k_i$, and we have that $n=\sum_{i=1}^n i k_i$. Then the symmetric product $\Xs$ admits a stratification with strata $X^{(n)}_{\nu}$ in one-to-one correspondence to such partitions $\nu=(n_1, \cdots, n_l)$ of $n$, defined by
$$X^{(n)}_{\nu}:=\left\{\sum_{i=1}^l n_i[x_i] \ \vert \ x_i \neq x_j, {\rm if} \ i \neq j \right\},$$
or, in terms of the sequence $\underbar{k}$ associated to the given partition $\nu$, 
$$X^{(n)}_{\nu}\cong \left( (\prod_{i=1}^n X^{k_i}) \setminus \Delta \right)/ S_{k_1} \times \cdots \times S_{k_n},$$
with $\Delta$ denoting the large diagonal in $ X^{\sum k_i}$, as before.

Let us now consider the summand $$\left( (\prod_i X^{k_i}) \setminus \Delta \right) \times \prod_i A_i^{k_i} / \prod_i S_{k_i} $$
of the coefficient of $t^n$ in the power structure (\ref{ps}), corresponding to a sequence $\underbar{k}$ of non-negative integers $\{k_i\}_{i > 0}$ so that $\sum_i i k_i=n$. If $\nu$ denotes the associated partition of $n$, let $\pi_{\nu}$ be the projection from the above summand onto  $\left((\prod_i X^{k_i}) \setminus \Delta \right) / \prod_i S_{k_i} =X^{(n)}_{\nu}$. Composing $\pi_{\nu}$ with the inclusion $i_{\nu}: X^{(n)}_{\nu} \hookrightarrow \Xs$ of the stratum into the symmetric product $\Xs$, we get a morphism 
\begin{equation}\label{extmot} \pi_{\underbar{k}}: = i_{\nu} \circ \pi_{\nu}: \left( (\prod_i X^{k_i}) \setminus \Delta \right) \times \prod_i A_i^{k_i} / \prod_i S_{k_i} \to \Xs.\end{equation}
The corresponding isomorphism class (up to decomposition) over $\Xs$ depends only on the
isomorphism classes (up to decomposition) of the $A_i$.
Finally, putting all partitions of $n$ together, the coefficient of $t^n$ in (\ref{ps}) can be now regarded as a well-defined  element in $S_0(var/\Xs)$. %resp. $K_0(var/\Xs)$.

Therefore, for a fixed variety $X$, %and for $F(-)=S_0(var/-)$ or  $K_0(var/-)$, 
we can now make sense of a {\it motivic exponentiation}:
\begin{equation}\label{exp}
(-)^X:1+tS_0(var/(\C)[[t]] \to PS_0(var/X):=\sum_{n\geq 0}  S_0(var/\Xs) \cdot t^n,
\end{equation}
defined by the same formula as $(\ref{ps})$, but keeping track of the strata of  symmetric products corresponding to each partition. This is a refinement of the corresponding geometric definition of the motivic power structure on $S_0(var/\C)$, which one gets back by
$$k_!\left((-)^X\right) = (-)^{[X]},$$
with $k: X\to pt$ the constant map, using the identifications $pt^{(n)}=pt$.\\

The exponentiation defined in (\ref{exp}) has the following properties, which can be  {\em directly deduced from the proof} of \cite{GLM0,GLM} for the corresponding properties of the geometric power structure of the motivic Grothendieck semi-group:
\begin{enumerate}
\item[(i')] $\left(A(t)\right)^\emptyset=1 \in PS_0(\emptyset)$,
\item[(ii')] $\left(A(t)\right)^{pt}=A(t)$, using the identifications $pt^{(n)}=pt$,
\item[(iii')] $\left( A(t) \cdot B(t) \right)^X=\left( A(t)  \right)^X \odot \left( B(t) \right)^X$,
\item[(iv')] $\left(A(t)\right)^{X}=i_!\left(A(t)\right)^{Y} \odot j_!\left(A(t)\right)^U$, with $i: Y\hookrightarrow X$ a closed
inclusion and $j: U \hookrightarrow  X$ the inclusion of the open complement $U:=X \setminus Y$.
\item[(v')] $\pi_!\left( (A(t) )^{X' \times X}\right)=\left( \left( A(t)  \right)^{[X']}  \right)^{X}$, for $\pi:X' \times X \to X$ the projection,
\item[(vi')] $(1+t)^X=1+[id_X]t+ {\rm higher \ order \ terms}$,
\item[(vii')] $\left( A(t^k)  \right)^X=P_k(\left( A(t)  \right)^X)$, with $P_k$ the $k$-th power operation.
\end{enumerate}
In the following we only use the properties (iii'), (v') and (vii') above.
\bigskip

Finally, to extend the above definition and properties to a motivic exponentiation on the Grothendieck group level:
\begin{equation}\label{expm}
(-)^X:1+tK_0(var/(\C)[[t]] \to PK_0(var/X):=\sum_{n\geq 0}  K_0(var/\Xs) \cdot t^n,
\end{equation}
we just use as in \cite{GLM0,GLM} the fact that any normalized power series $A(t) \in 1+tK_0(var/(\C)[[t]]$ can be factored as a quotient $A=B \cdot C^{-1}$, with $B, C$ in the image of the  canonical semi-ring map
$${\rm can}: (1+tS_0(var/(\C)[[t]], \cdot) \to (1+tK_0(var/(\C)[[t]], \cdot)$$ so that $$(A(t))^X:=(B(t))^X \cdot \left
((C(t))^X\right)^{-1}$$ gives a  well-defined exponentiaton on the Grothendieck group level. Here we use that $(C(t))^X$ is by definition a normalized power series in the Pontrjagin ring $PK_0(var/X)$, so it can be inverted. More precisely, this method shows that there exists at most one such exponentiation on the Grothendieck group level. The existence of this exponentiation can be translated into the existence of a group homomorphism $$(-)^X:(1+tK_0(var/(\C)[[t]], \cdot) \to (PK_0(var/X), \odot)$$ making the 
following diagram commutative:
\begin{equation}\label{diag}
\begin{CD}
(tS_0(var/(\C)[[t]],+) @>{\rm can} >> (tK_0(var/(\C)[[t]],+) \\
@V{\rm Exp} VV      @V\wr V{\rm Exp} V \\
(1+tS_0(var/(\C)[[t]],\cdot) @>{\rm can} >> (1+tK_0(var/(\C)[[t]],\cdot) \\
@V{(-)^X} VV      @VV{(-)^X} V \\
(PS_0(X), \odot) @>{\rm can} >> (PK_0(X), \odot).
\end{CD}
\end{equation}
The semi-group homomorphism  ${\rm Exp}$ on the left hand side of the diagram is defined as in (\ref{Euler-prod}), but using the geometric definition of the motivic power structure (instead of the pre-lambda ring structure). The left motivic exponentiation $(-)^X$ is already defined as a semi-group homomorphism.
The commutativity of the upper square of the diagram follows from the geometric interpretation of the pre-lambda ring structure in terms of Kapranov zeta function. So the claim follows now from the fact that $ (tK_0(var/(\C)[[t]],+) $ is the Grothendieck group completion of the semi-group $(tS_0(var/(\C)[[t]],+) $. Properties $(i')-(vii')$ for this extended exponentiation follow directly from the corresponding properties on the semi-group level.\\

We conclude this section with the following examples:
\begin{example}\label{ex-s}\rm  The following identity refines the one from Example \ref{ex1}:
\begin{equation}
 (1-t)^{-X}:=((1-t)^{-1})^X=1+\sum_{n=1}^{\infty}[id_{\Xs}] \cdot t^n
\end{equation}
Indeed, in this case $A_i=[pt]$, for all $i$, and each projection $\pi_{\nu}$ in the above construction can be identified %to 
with
$id_{X^{(n)}_{\nu}}$. It follows that $[\pi_{\underbar{k}}]=[i_{\nu}]$, and the result follows by summing up over all 
the 
partitions $\nu$ of $n$.
\end{example} 

\begin{example}\label{ex-conf}\rm  The following identity refines the one  from Example \ref{ex2}:
\begin{equation}
 (1+t)^{X}=1+\sum_{n=1}^{\infty}[X^{\{n\}} \overset{i_n}{\to}{\Xs}] \cdot t^n=1+\sum_{n=1}^{\infty}{(i_n)}_![id_{X^{\{n\}}}]  \cdot t^n,
\end{equation}
with $i_n:X^{\{n\}} \hookrightarrow \Xs$ the inclusion of the configuration space $X^{\{n\}}$ of $n$ unlabeled points on $X$ into the symmetric product $\Xs$. Note that $i_n$ corresponds to the partition $\nu={(1,\cdots,1)}$ of $n$.
\end{example} 

\begin{example}\label{ex-sym}\rm For $\pi:X' \times X \to X$ the projection, property $(v')$ and
Example \ref{ex-s} yield:
\begin{equation}
1+\sum_{n=1}^{\infty}[(X'\times X)^{(n)} \to{\Xs}] \cdot t^n =
\pi_!\left( (1-t)^{-{X' \times X}}\right)=\left( \left( 1-t \right)^{-[X']}  \right)^{X}
\end{equation}
\end{example}

The following refinement of Theorem \ref{Hilbert-GLM} is just a reformulation of
\cite[\S 2.1]{BBS} in terms of our new motivic exponentiation. It will be the starting point for the proof of Theorem \ref{thm1}.

\begin{theorem}[Behrend, Bryan,  Szendr\"oi] \label{Hilbert-BBS} Let $X$ be a smooth and pure $d$-dimensional complex quasi-projective variety. Then 
\begin{equation}\label{main}
1+\sum_{n \geq 1} [X^{[n]}\overset{\pi_n}{\to} \Xs] \cdot t^n=\left(1+\sum_{n \geq 1} \left[{\rm Hilb}^n_{\C^d,0}\right] \cdot t^n\right)^{X}.
\end{equation}
\end{theorem}

%%%%%%%%%%%%%%%%%%%%%%%%%%%%%%%%
%%%%%%%%%%%%%%%%%%%%%%%%%%%%%%%%

\section{Proof of the main result}

%%%%%%%%%%%%%%%%%%%%%%%%%%%%%%%%

\subsection{Motivic Hirzebruch classes}\label{motHir}
The un-normalized Hirzebruch class transformation 
$$T_{y*}:K_0(var/X) \to H_*(X):=H^{BM}_{even}(X) \otimes \Q[y]$$
was introduced in \cite{BSY} as a class version of the virtual Hodge polynomial $\chi_y$. Its normalization, denoted here by $\widehat{T}_{y*}$, provides a 
functorial unification of of the Chern class transformation of MacPherson \cite{MP}, Todd class transformation of Baum-Fulton-MacPherson \cite{BFM} and L-class transformation of Cappell-Shaneson \cite{CS1}, respectively. This normalization $\widehat{T}_{y*}$ is obtained by pre-composition the transformation $T_{y*}$ with the normalization functor
$$\Psi_{(1+y)}:H^{BM}_{even}(X) \otimes \Q[y] \to H^{BM}_{even}(X) \otimes \Q[y,(1+y)^{-1}]$$ given in degree $2k$ by multiplication by $(1+y)^{-k}$. And it follows from \cite[Theorem 3.1] {BSY} that $\widehat{T}_{y*}:=\Psi_{(1+y)} \circ T_{y*}$ takes in fact values in $H^{BM}_{even}(X) \otimes \Q[y]$, so one is allowed to specialize the parameter $y$ to the value $y=-1$.

The transformations $T_{y*}$ and  $\widehat{T}_{y*}$ are functorial for proper push-forwards, and they commute with exterior products. If $X$ is a point, these transformations reduce to the pre-lambda ring homomorphism $$\chi_y:K_0(var/\C) \to \Z[y].$$
Recall that a pre-lambda ring homomorphism commutes with the corresponding Euler products.

The un-normalized {\it motivic Hirzebruch class} of a complex algebraic variety $X$ is defined by:
$$T_{y*}(X):=T_{y*}([id_X]).$$ Similarly, we define the normalized motivic Hirzebruch class of $X$ by using instead the transformation $\widehat{T}_{y*}$. If $X$ is smooth, then $T_{y*}(X)$ is Poincar\'e dual to the Hirzebruch cohomology class $T^*_{y}(TX)$ appearing in the generalized Hirzebruch-Riemann-Roch theorem \cite{H}, and which in Hirzebruch's philosophy corresponds to the non-characteristic power series:
\begin{equation} Q_y(\alpha):=\frac{\alpha (1+ye^{-\alpha})}{1-e^{-\alpha}} \in \Q[y][[\alpha]],
\end{equation}
with $Q_y(0)=1+y$. 
More precisely, 
\begin{equation}
T^*_{y}(TX):=\prod_{i=1}^{{\rm dim} X} Q_y(\alpha_i),
\end{equation}
with $\{\alpha_i\}$ the formal Chern roots of the  holomorphic tangent bundle $TX$ of $X$.
The associated normalized (or characteristic) power series is \begin{equation}\widehat{Q}_y(\alpha):=\frac{Q_y(\alpha(1+y))}{1+y}=\frac{\alpha(1+y)}{1-e^{-\alpha(1+y)}}-\alpha y \ ,\end{equation} which defines the normalized cohomology Hirzebruch class $\widehat{T}^*_y(-)$. 
By specializing the parameter $y$ of $\widehat{T}^*_y(-)$ to the three distinguished values $y=-1, 0$ and $1$, we recover the cohomology Chern, Todd, and $L$-class, respectively.
Also, if $X$ is smooth, the classes $\widehat{T}_{y*}(X)$ and $\widehat{T}^*_y(TX)$ are Poincar\'e dual to each other (cf. \cite{BSY}). Similarly, the un-normalized cohomology Hirzebruch class ${T}^*_y(-)$ specializes for the three distinguished values $y=-1, 0$ and $1$ to the top Chern class, Todd class and Atiyah-Singer $L$-class $\widetilde{L}^*$, respectively, with $\Psi_2 \widetilde{L}^* =L^*$ the Hirzebruch $L$-class (where $\Psi_2$ denotes the corresponding cohomological Adams operation defined in degree $2k$ by multiplication with $2^k$).\\

It follows from \cite{BSY} that, even if $X$ is singular, by specializing to $y=-1$ one gets that
\begin{equation}\label{idc}
\widehat{T}_{-1*}(X)=c_*(X) \otimes \Q 
\end{equation} 
is the rationalized homology Chern class of MacPherson \cite{MP}. For the un-normalized Hirzebruch class, we only get the degree-zero part of the MacPherson Chern class as its specialization at $y=-1$:
\begin{equation}
{T}_{-1*}(X)=c_0(X) \otimes \Q 
\end{equation} 
This motivates our discussion in Section \ref{MPChern}.

\begin{remark}\label{rem-main}\rm
We will denote by the same symbol, $T_{y*}(-)$, the induced functorial transformation $PK_0(var/X) \to PH_*(X)$. 
Since the un-normalized motivic Hirzebruch class transformation $$T_{y*}:K_0(var/X) \to H_*(X):=H^{BM}_{even}(X) \otimes \Q[y]$$ commutes with exterior products and proper push-forward, it follows that it also commutes with the Pontrjagin product $\odot$, hence
$$T_{y*}(-): PK_0(var/X) \to PH_*(X)$$ becomes a ring homomorphism.
%And since 
Since 
the diagonal embeddings $p_k:X^n \to X^{nk}$ (hence also the induced maps $p_k^{(n)}$) are proper, $T_{y*}(-)$ also commutes with the power operations $P_k$ defined on $PK_0(var/X)$ and $PH_*(X)$, respectively.
\end{remark}

%%%%%%%%%%%%%%%%%%%%%%%%%%%%%%%%%

\subsection{Hirzebruch classes of symmetric products and configuration spaces}
A generating series for the un-normalized motivic Hirzebruch classes of symmetric products of a quasi-projective variety $X$ was given in \cite{CMSSY}. % by using Saito's theory of algebraic mixed Hodge modules. 
A motivic reformulation of this formula, obtained by using Example \ref{ex-sym}, is the following:
\begin{theorem}\label{mot} Let $X$ and $X'$ be quasi-projective complex algebraic varieties.
Then:
\begin{equation}\label{eqmot}
{T_{(-y)}}_*\left(  \left( (1-t)^{-[X']} \right)^X \right)=(1-t \cdot d_*)^{-\chi_{-y}(X') \cdot {T_{(-y)}}_*(X)},
\end{equation}
with $$(1-t \cdot d_*)^{-(\cdot)}:=\exp \left(  \sum_{r=1}^{\infty} \Psi_r d^r_*(\cdot) \frac{t^r}{r}  \right) :H_*(X) \to PH_*(X).$$
Here $d^r:X \to X^{(r)}$ is the composition of the diagonal embedding $X \to X^r$ with the natural projection $X^r \to X^{(r)}$, and $\Psi_r$ denotes the $r$-th homological Adams operation which is defined by multiplication by $\frac{1}{r^k}$ on $H_{2k}^{BM}(-;\Q)$ and by sending $y$ to $y^r$.
\end{theorem}

\begin{cor} Let $X$ be a quasi-projective variety and $\alpha \in K_0(var/\C)$ be a fixed virtual class.
Then, in the above notations, we obtain:
 \begin{equation}\label{eqmot2}
{T_{(-y)}}_*\left(  \left( (1-t)^{-\alpha} \right)^X \right)=(1-t \cdot d_*)^{-\chi_{-y}(\alpha) \cdot {T_{(-y)}}_*(X)},
\end{equation}
\end{cor}

\begin{proof} We use that $\alpha=[X']-[X'']$, with $X'$ and $X''$ quasi-projective varieties, together with the fact that ${T_{(-y)}}_*:PK_0(var/X) \to PH_*(X)$ and $\chi_{-y}:K_0(var/\C) \to \Z[y]$ are ring homomorphisms.
\end{proof}

%\noindent(Our definition for the Adams operation $\Psi_r$ differs %than 
%from
%the one used in \cite{CMSSY}, where the action on homology %was 
%is
%linearly extended to the polynomial coefficients.)

In particular, %we get 
for $\alpha=[pt]$ %and by using Example \ref{ex-s} 
we get the following result (see also \cite[Corollary 1.2]{CMSSY})
by using Example \ref{ex-s}  :
\begin{cor}\label{symHir} If $X$ is a quasi-projective complex algebraic variety, then:
\begin{eqnarray*}
\sum_{n\geq0} {T_{(-y)}}_*(\Xs) \cdot t^n&=& {T_{(-y)}}_*((1-t)^{-X})\\ &=&(1-t \cdot d_*)^{- {T_{(-y)}}_*(X)} \\ &=& \exp \left(  \sum_{r=1}^{\infty} \Psi_r d^r_*\left( {T_{(-y)}}_*(X)\right) \frac{t^r}{r}  \right).
\end{eqnarray*}
\end{cor}

Moreover, Example \ref{ex-conf} can be used to derive the following:

\begin{prop}\label{conf} For a quasi-projective complex algebraic variety $X$, 
let $i_n:X^{\{n\}} \hookrightarrow \Xs$ denote as above the inclusion of the configuration space $X^{\{n\}}$ of $n$ unlabeled points on $X$ into the symmetric product $\Xs$. Then the following generating series formula holds:
\begin{equation}
\sum_{n\geq0} {T_{(-y)}}_*([i_n]) \cdot t^n=(1-t^2 \cdot d^2_*)^{{T_{(-y)}}_*(X)} \odot (1-t \cdot d_*)^{- {T_{(-y)}}_*(X)} \ .
\end{equation}
\end{prop}

\begin{proof} By applying the ring homomorphism 
$${T_{(-y)}}_*:PK_0(var/X) \to PH_*(X)$$
to the identity in Example \ref{ex-conf}, and using the fact that $1+t=\frac{1-t^2}{1-t}$, we obtain:
% {\allowdisplaybreaks
\begin{eqnarray*}\sum_{n\geq0} {T_{(-y)}}_*([i_n]) \cdot t^n &=& {T_{(-y)}}_*((1+t)^X) \\ 
&\overset{(iii')}{=} & {T_{(-y)}}_*\left((1-t^2)^X \odot  (1-t)^{-X}\right) \\
&=&{T_{(-y)}}_*\left((1-t^2)^X\right) \odot  {T_{(-y)}}_*\left((1-t)^{-X}\right)\\
&\overset{(vii')}{=} & {T_{(-y)}}_*\left(P_2\left((1-t)^X\right)\right) \odot  {T_{(-y)}}_*\left((1-t)^{-X}\right)\\
&=&P_2\left( {T_{(-y)}}_*\left((1-t)^X\right) \right) \odot  {T_{(-y)}}_*\left((1-t)^{-X}\right)
\end{eqnarray*}
%}
The desired result follows now from formula (\ref{eqmot}).
\end{proof}

The result of Proposition \ref{conf} should be regarded as a characteristic class version of Getzler's generating series formula for the virtual Hodge polynomial (or, more generally, for the Hodge-Deligne polynomial) of configuration spaces $X^{\{n\}}$ of $n$ unlabeled points on $X$, see \cite[Corollary 5.7]{Ge}.

%%%%%%%%%%%%%%%%%%%%%%%%%%%%%%%%%

\subsection{Proof of Theorem \ref{thm1}}
After developing $(\ref{main})$ into the corresponding Euler product with exponents $\alpha_k \in K_0(var/\C)$, and by using the rules of exponentiation, we get:
\begin{eqnarray*}
1+\sum_{n \geq 1} [X^{[n]}\overset{\pi_n}{\to} \Xs] \cdot t^n &=&\left(1+\sum_{n \geq 1} \left[{\rm Hilb}^n_{\C^d,0}\right] \cdot t^n\right)^{X}\\ &=&\left( \prod_{k=1}^{\infty} (1-t^k)^{-\alpha_k}\right)^{X} \\ &\overset{(iii')}{=}& \bigodot_{k=1}^{\infty} \left( (1-t^k)^{-\alpha_k} \right)^X
\\ &\overset{(vii')}{=}& \bigodot_{k=1}^{\infty} P_k \left( \left( (1-t)^{-\alpha_k} \right)^X \right)
\end{eqnarray*}
where $\bigodot$ denotes %here 
the Pontrjagin product in the motivic Pontrjagin ring $PK_0(var/X)$.
We next apply the ring homomorphism 
$${T_{(-y)}}_*:PK_0(var/X) \to PH_*(X)$$
to the above identity and use that fact that this transformation commutes with proper push-forwards, with the Pontrjagin multiplication, and with the power operations $P_k$, to obtain
the following:
 {\allowdisplaybreaks
\begin{eqnarray*}
1+\sum_{n \geq 1} {\pi_n}_*{T_{(-y)}}_*(X^{[n]})  \cdot t^n &=& {T_{(-y)}}_*\left(  \bigodot_{k=1}^{\infty} P_k \left( \left( (1-t)^{-\alpha_k} \right)^X \right) \right) \\
&=&  \bigodot_{k=1}^{\infty} P_k  \left( {T_{(-y)}}_* \left( \left( (1-t)^{-\alpha_k} \right)^X \right)  \right)\\ 
&\overset{(\ref{eqmot2})}{=}&  \bigodot_{k=1}^{\infty} P_k  \left( (1-t\cdot d_*)^{-\chi_{-y}(\alpha_k) \cdot  {T_{(-y)}}_*(X)}  \right)\\ 
&{=:}&  \bigodot_{k=1}^{\infty}  \   (1-t^k\cdot d^k_*)^{-\chi_{-y}(\alpha_k) \cdot  {T_{(-y)}}_*(X)}  \\ 
&=:& \left(1+ \sum_{n=1}^{\infty} \chi_{-y}({\rm Hilb}^n_{\C^d,0}) \cdot t^n \cdot d^n_* \right)^{{T_{(-y)}}_*(X)} 
\end{eqnarray*}
}
\hfill$\square$

Note that the method of proof of the above result yields the following characteristic class version of Proposition \ref{prop1}:
\begin{theorem}\label{mt2} Let $X$ be a quasi-projective variety and $1+\sum_{n \geq 1} A_n t^n \in K_0(var/\C)[[t]]$ be a normalized power series. Then:
$${T_{(-y)}}_* \left( (1+\sum_{n \geq 1} A_n t^n )^X\right)=\left(1+ \sum_{n=1}^{\infty} \chi_{-y}(A_n)\cdot t^n \cdot d^n_* \right)^{{T_{(-y)}}_*(X)} .$$
\end{theorem}

%\bigskip

%%%%%%%%%%%%%%%%%%%%%%%%%%%%%%%%%%

\section{Hirzebruch classes of virtual motives of Hilbert schemes of threefolds}
In this section we show that our methods can be combined with the approach of \cite{BBS}, where virtual motives of Hilbert schemes of threefolds are defined in relation with Donaldson-Thomas invariants.

Let $f:M \to \C$ be a regular function on a smooth quasi-projective variety, with singular locus $$Z=\{df=0\} \subset M.$$ The {\it relative virtual motive} of $Z$ is defined as in \cite{BBS} by the formula $$[Z]_{\rm relvir} = -\LB^{-\frac{{\rm dim} M}{2}} [\phi_f]_Z \in \K_0(var/Z)[\LB^{-1/2}],$$ in terms of motivic vanishing cycles $\phi_f$ (where we ignore the monodromy action of roots of unity, which is not needed here). This is a motivic refinement of Behrend's constructible function $\nu_Z$ constructed in \cite{Beh}.
Pushing down to a point, we get  the {\it virtual motive} $$[Z]_{\rm vir} = -\LB^{-\frac{{\rm dim} M}{2}} [\phi_f] \in \K_0(var/\C)[\LB^{-1/2}].$$
A priori, these (relative) virtual motives may depend on the chosen equation $f$ with degeneracy locus $Z$.

In particular, this construction applies to the Hilbert scheme $(\C^3)^{[n]}$, which can be realized as a degeneracy locus as above after a choice of a volume form on $\C^3$, see \cite{BBS}[Prop.2.1] for more details. By restriction to the fiber above $0 \in \C^3$, one gets a {\it virtual motive} for the punctual Hilbert scheme $$[{\rm Hilb}^n_{\C^3,0}]_{\rm vir} \in \K_0(var/\C)[\LB^{-1/2}].$$
As shown in \cite{BBS}, the definition of these virtual motives is independent of the choice of the volume form on $\C^3$.

Moreover, the proof of \cite{BBS}[Prop.2.6] gives the following generating series for the push-forwards under the Hilbert-Chow morphisms of the (relative) virtual motives of $(\C^3)^{[n]}$:
\begin{equation}\label{mainb}
1+\sum_{n \geq 1} {\pi_n}_*[(\C^3)^{[n]}]_{\rm relvir} \cdot t^n=\left(1+\sum_{n \geq 1} \left[{\rm Hilb}^n_{\C^3,0}\right]_{\rm vir} \cdot t^n\right)^{\C^3}.
\end{equation}
This is a relative version of the following virtual counterpart of Thm.\ref{Hilbert-GLM} (see \cite{BBS}[Prop.3.2]):
\begin{equation}\label{mainc}
1+\sum_{n \geq 1} [(\C^3)^{[n]}]_{\rm vir} \cdot t^n=\left(1+\sum_{n \geq 1} \left[{\rm Hilb}^n_{\C^3,0}\right]_{\rm vir} \cdot t^n\right)^{[\C^3]}.
\end{equation}

Note that in formula (\ref{mainb}) we need to use the extension of our {\it geometric} motivic exponentiation $(A(t))^X$ to normalized power series $A(t) \in K_0(var/\C)[\LB^{-1/2}][[t]]$. This extension will be constructed below in two steps. \\

Let us first explain this {\it geometric} extension at the level of Grothendieck groups. Let $S_0^{\pm}(var/X)$ be the relative semi-ring of $\Z_2$-graded quasi-projective varieties over the fixed  quasi-projective variety $X$, which is defined as the semi-ring $S_0(var/X)$ using $\Z_2$-graded varieties, i.e., $Z=Z_0 \uplus Z_1$ with the commutative (graded) product and  external product $\boxtimes$ induced from the usual fiber and cartesian product. Note that there are canonical functorial semi-ring homomorphisms $${\rm can}: S_0(var/X) \to S_0^{\pm}(var/X) \to K_0(var/X)$$ defined by
$$[Z_0 \to X] \mapsto [Z_0 \uplus \emptyset \to X] \ \ \ \  {\rm and} \ \ \ \ 
 [Z_0 \uplus Z_1 \to X] \mapsto [Z_0 \to X] - [Z_1 \to X],$$
with
$$ [Z_0 \uplus Z_1 \to X] \boxtimes  [Z'_0 \uplus Z'_1 \to X'] \mapsto \left( [Z_0 \to X] - [Z_1 \to X] \right) \boxtimes \left( [Z'_0 \to X'] - [Z'_1 \to X'] \right).$$
 The first homomorphism above is an injection, whereas the second one is an epimorphism. 
Similar considerations apply to equivariant situations, which we only need for $X$ a point space, in the context of equivariant external products of ($\Z_2$-graded and virtual) varieties. More precisely, we have the following result (for a proof see \cite{BBS}[Lem.1.4]):
\begin{lemma} There exist canonical equivariant external products $\boxtimes^n$ induced from the cartesian products of varieties, which fit into a commutative diagram:
\begin{equation}
\begin{CD}
S_0(var/\C) @>\boxtimes^n >> S_0^{S_n}(var/\C) \\
@VVV @VVV\\
S_0^{\pm}(var/\C) @>\boxtimes^n >> S_0^{S_n,^{\pm}}(var/\C) \\
@VVV @VVV\\
K_0(var/\C) @>\boxtimes^n >> K_0^{S_n}(var/\C) .
\end{CD}
\end{equation}
\end{lemma}
Using these equivariant external products, we can extend our {\it geometric} definition of a motivic exponentiation also  to normalized power series $1+\sum_{i \geq 1} A_i t^i$ with either $\Z_2$-graded  coefficients (i.e., $A_i \in S_0^{\pm}(var/\C)$) or with virtual coefficients (i.e. $A_i \in K_0(var/\C)$). In this way, we get a commutative diagram of semi-group homomorphisms:
\begin{equation}
\begin{CD}
(1+ tS_0(var/\C)[[t]],\cdot) @>{(-)^X} >> (PS_0(X),\odot) \\
@VVV @VVV\\
(1+ tS_0^{\pm}(var/\C)[[t]],\cdot) @>{(-)^X} >> (PS^{\pm}_0(X),\odot) \\
@VVV @VVV\\
(1+ tK_0(var/\C)[[t]],\cdot) @>{(-)^X} >> (PK_0(X),\odot).
\end{CD}
\end{equation}
Note that the top two exponentiation maps are semi-group homomorphisms by the construction of the geometric motivic power structure given in \cite{GLM0,GLM}, which directly applies also to the $\Z_2$-graded context. The bottom exponentiation in the above diagram is then also a semi-group homomorphism by the surjectivity of the left bottom vertical arrow. By uniqueness, this {\it geometrically} defined exponentiation
$$(-)^X: (1+ tK_0(var/\C)[[t]],\cdot) \to  (PK_0(X),\odot)$$  has to agree with the exponentiation constructed already in Section \ref{motexp}, since they coincide on (the image of) $(1+ tS_0(var/\C)[[t]],\cdot)$.

In the second step, we extend our geometric definition of a motivic exponentiation to a normalized power series %$1+\sum_{i\geq 1} A_i t^i$, 
with localized coefficients in $K_0(var/\C)[\LB^{-1/2}]$.
This can be defined as in \cite{GLM0,GLM,BBS} by modifying (\ref{extmot}) using  (for $A_i \in K_0(var/\C)$ and $c_i \in \Z$) the identities: 
\begin{multline*} \left( (\prod_i X^{k_i}) \setminus \Delta \right) \times \prod_i \left((-\LB^{1/2})^{c_i}A_i\right)^{k_i} / \prod_i S_{k_i} \\ =(-\LB^{1/2})^{\sum_i c_ik_i} \left( (\prod_i X^{k_i}) \setminus \Delta \right) \times \prod_i A_i^{k_i} / \prod_i S_{k_i} \in K_0(var/\Xs)[\LB^{-1/2}].
\end{multline*}

Altogether, we get a {\it geometric} motivic exponentiation:
$$(-)^X: (1+ tK_0(var/\C)[\LB^{-1/2}][[t]],\cdot) \to  (PK_0(X)[\LB^{-1/2}],\odot)$$ 
on the Pontrjagin ring  associated to the covariant functor $F(-)=K_0(var/-)[\LB^{-1/2}]$.
By construction, besides the usual rules of our motivic exponentiation, we also have the following: \begin{equation}\label{99}
(A(-t))^X=(A(t))^X \vert_{t \mapsto -t}.
\end{equation}

Note that Theorem \ref{mt2} also holds for power series with coefficients in the localized Grothendieck ring $K_0(var/\C)[\LB^{-1/2}]$. Here we choose the convention $$\chi_{-y}(-\LB^{1/2}):=y^{1/2} \ \ \ \ {\rm and} \ \ \ \ \Psi_r(y^{1/2}):=y^{r/2},$$ which fits for $y=1$ with the convention $\chi(\LB^{1/2})=-1$ used in \cite{BBS}.

By applying the un-normalized Hirzebruch class transformation to formula (\ref{mainb}) we get by (this extension of) Theorem \ref{mt2}  the following result:
\begin{theorem} In the above notations, the following formula holds:
\begin{equation}\label{100}
\sum_{n=0}^{\infty} {\pi_n}_*{T_{(-y)}}_*([(\C^3)^{[n]}]_{\rm relvir}) \cdot t^n = \left(1+ \sum_{n=1}^{\infty} \chi_{-y}([{\rm Hilb}^n_{\C^3,0}]_{\rm vir}) \cdot t^n \cdot d^n_* \right)^{{T_{(-y)}}_*(\C^3)}
\end{equation}
Moreover, 
\begin{equation}\label{101}
\sum_{n=0}^{\infty} {\pi_n}_*{T_{(-y)}}_*([(\C^3)^{[n]}]_{\rm relvir}) \cdot (-t)^n =  \prod_{k=1}^{\infty} (1-t^k \cdot d^k_*)^{-\chi_{-y}(\alpha_k) \cdot {T_{(-y)}}_*(\C^3)},
\end{equation}
where the coefficients $\alpha_k \in K_0(var/{\C})[\LB^{-1/2}]$ of the corresponding Euler product are explicitly given by 
$$\alpha_k=\LB^{-3} \cdot (-\LB^{1/2})^{4-k} \cdot \frac{1-(-\LB^{1/2})^{2k}}{1-\LB}=\frac{(-\LB^{1/2})^{-k}-(-\LB^{1/2})^{k}}{\LB (1-\LB)}.$$
\end{theorem}
 
Note that in the above result, ${T_{(-y)}}_*(\C^3)=[\C^3] \in H^{BM}_{even}(\C^3)$ and the $\alpha_k$'s are computed  in the proof of \cite{BBS}[Thm.3.3]. For this explicit computation it is important to switch from the variable $t$ to $-t$ using (\ref{99}).\\

If one wants to extend the above results from $\C^3$ to an arbitrary smooth quasi-projective threefold $X$, one is faced with the problem that the Hilbert scheme $X^{[n]}$ is not known to be in general a global degeneracy locus, so the (relative) virtual motive $[X^{[n]}]_{\rm relvir}$ is not defined. Even so, the corresponding global Behrend constructible function is well-defined, see \cite{Beh}. However, by imitating the result of Theorem \ref{Hilbert-BBS}, one can still define the push-forward
$${\pi_n}_*[X^{[n]}]_{\rm relvir} \in K_0(var/\Xs)[\LB^{-1/2}]$$
by the generating series formula:
\begin{equation}\label{maine}
1+\sum_{n \geq 1} {\pi_n}_*[X^{[n]}]_{\rm relvir} \cdot t^n:=\left(1+\sum_{n \geq 1} \left[{\rm Hilb}^n_{\C^3,0}\right]_{\rm vir} \cdot t^n\right)^{X} \in PK_0(X)[\LB^{-1/2}].
\end{equation}
This is a relative version of the following formula established in \cite{BBS}[Prop.3.2]:
\begin{equation}\label{mainf}
1+\sum_{n \geq 1} [X^{[n]}]_{\rm vir} \cdot t^n=\left(1+\sum_{n \geq 1} \left[{\rm Hilb}^n_{\C^3,0}\right]_{\rm vir} \cdot t^n\right)^{[X]} \in PK_0(pt)[\LB^{-1/2}].
\end{equation}
By applying the un-normalized Hirzebruch class transformation to formula (\ref{maine}) we get by (the extension of) Theorem \ref{mt2}  the following result:
\begin{theorem}\label{mainl} For any smooth quasi-projective threefold $X$ the following formula holds:
\begin{equation}\label{110}
\sum_{n=0}^{\infty} {T_{(-y)}}_*({\pi_n}_*[X^{[n]}]_{\rm relvir}) \cdot t^n = \left(1+ \sum_{n=1}^{\infty} \chi_{-y}([{\rm Hilb}^n_{\C^3,0}]_{\rm vir}) \cdot t^n \cdot d^n_* \right)^{{T_{(-y)}}_*(X)}
\end{equation}
Moreover, 
\begin{equation}\label{111}
\sum_{n=0}^{\infty} {T_{(-y)}}_*({\pi_n}_*[X^{[n]}]_{\rm relvir}) \cdot (-t)^n =  \prod_{k=1}^{\infty} (1-t^k \cdot d^k_*)^{-\chi_{-y}(\alpha_k) \cdot {T_{(-y)}}_*(X)},
\end{equation}
with coefficients $\alpha_k \in K_0(var/{\C})[\LB^{-1/2}]$ given as before  by 
$$\alpha_k=\frac{(-\LB^{1/2})^{-k}-(-\LB^{1/2})^{k}}{\LB(1-\LB)}.$$
\end{theorem}
Note that the identification
$${T_{(-y)}}_*({\pi_n}_*[X^{[n]}]_{\rm relvir}) = {\pi_n}_*{T_{(-y)}}_*([X^{[n]}]_{\rm relvir})$$ does not make sense except for $X=\C^3$. Nevertheless, a suitable specialization at $y=1$ is well-defined and reduces to the functoriality of the MacPherson Chern classes of Behrend's constructible function of the Hilbert scheme:
\begin{equation}\label{112}c_*({\pi_n}_*\nu_{X^{[n]}}) = {\pi_n}_*{c}_*(\nu_{X^{[n]}}).\end{equation}
This will be explained in the next section.

%%%%%%%%%%%%%%%%%%%%%%%%%%%%%%%%%%%%

\section{MacPherson-Chern classes}\label{MPChern} In order to explain (\ref{112}) and the formula of Corollary \ref{cormain2} about the MacPherson-Chern classes of Hilbert schemes, we need to say a few words about the normalized version of our main result in Theorem \ref{thm1}, as well as in Theorems \ref{mt2} and \ref{mainl}. 

Recall that the normalized homology Hirzebruch class $\widehat{T}_{y*}$ is defined as  $$\widehat{T}_{y*}:=\Psi_{(1+y)} \circ T_{y*},$$ with 
$$\Psi_{(1+y)}:H^{BM}_{even}(X) \otimes \Q[y] \to H^{BM}_{even}(X) \otimes \Q[y,(1+y)^{-1}]$$ the normalization functor given in degree $2k$ by multiplication by $(1+y)^{-k}$. Moreover, by letting $y=-1$, we get that $$\widehat{T}_{-1*}(X)=c_*(X) \otimes \Q$$
is the rationalized homology Chern class of MacPherson. \\

By applying the normalization 
functor ${\Psi_{(1-y)}}$ (note that due to our indexing conventions, $y$ is replaced here by $-y$) to the left-hand side of the formula in Theorem \ref{thm1},  we get the generating series $\sum_{n\geq 0} \pi_{n*}\widehat{T}_{(-y)_*}(X^{[n]}) \cdot t^n$, which at $y=1$ yields the left-hand side of the Chern class formula of Corollary  \ref{cormain2}. 
Applying the same procedure to the right-hand side of 
the identity in Theorem \ref{thm1}, we first note that the normalization functor ${\Psi_{(1-y)}}$ commutes with push-forward for proper maps, as well as with exterior products, therefore ${\Psi_{(1-y)}}$ commutes with the Pontrjagin product (hence with the exponential) and it also commutes with the power operations $P_k$ on the homology Pontrjagin ring $PH_*(X)$.   
Moreover, as shown in \cite[Lemma 4.2]{CMSSY}, the following identification of transformations holds:
\begin{equation}\label{cmssy}
\lim_{y \to 1}  {\Psi_{(1-y)}}\Psi_r{{T_{(-y)}}_*}(-)=\widehat{T}_{-1*}(-) = c_*(-) \otimes \Q
:K_0(var/X) \to H^{BM}_{even}(X;\Q). 
\end{equation}
So by applying the identity (\ref{cmssy}) to the distinguished element $[id_X]\in K_0(var/X)$, we obtain that: 
\begin{equation}
\lim_{y \to 1}  {\Psi_{(1-y)}}\Psi_r{{T_{(-y)}}_*}(X)=\widehat{T}_{-1*}(X)=c_*(X) \otimes \Q . 
\end{equation}
In other words, after specializing the parameter $y$ % at 
to
the value $y=1$, the normalization functor  ${\Psi_{(1-y)}}$ ``cancels out" the Adams operation $\Psi_r$. Corollary \ref{cormain2} follows now readily.\\

The same argument yields the following counterparts of Theorems \ref{mt2} and \ref{mainl} for the MacPherson Chern class transformation $c_*:K_0(var/-) \to H^{BM}_{ev}(-;\Q)$ (for a definition of this motivic lift of $c_*$, see \cite{BSY}).
\begin{theorem}\label{mt3} Let $X$ be a quasi-projective variety and $1+\sum_{n \geq 1} A_n t^n \in K_0(var/\C)[[t]]$ 
or in $K_0(var/\C)[\LB^{-1/2}][[t]]$ be a normalized power series. Then:
$$c_* \left( (1+\sum_{n \geq 1} A_n t^n )^X\right)=\left(1+ \sum_{n=1}^{\infty} \chi(A_n)\cdot t^n \cdot d^n_* \right)^{c_*(X)} .$$
\end{theorem}
 
 Note that the MacPherson Chern class transformation $c_*$ factorizes over the group of (algebraically) constructible functions $CF(-)$ as 
 $$K_0(var/-)[\LB^{-1/2}] \overset{e}{\to} CF(-) \overset{c_*}{\to} H^{BM}_{ev}(-;\Q),$$
 with the canonical transformation $e$ defined in \cite{BSY} and $e(\LB^{-1/2}):=-1$ (to fit with the convention of \cite{BBS}). 
 Recall that the {\it Aluffi class} of the Hilbert scheme $X^{[n]}$ is defined in \cite{Beh} as  the MacPherson Chern class of the corresponding Behrend function:
 $$c_*^A(X^{[n]}):=c_*(\nu_{X^{[n]}}).$$
 Moreover,
 \begin{equation}\label{119} 
 e({\pi_n}_*[X^{[n]}]_{\rm relvir}) = {\pi_n}_*(\nu_{X^{[n]}}),
 \end{equation}
 as it follows from the compatibility of the transformation $e$ with vanishing cycles in the motivic and resp. constructible function context, see \cite{CMSS,Sch}. With these identifications, we obtain as a corollary of Theorem \ref{mainl}:
\begin{cor}\label{mainm} For any smooth quasi-projective threefold $X$ the following formula holds:
\begin{equation}\label{120}
\sum_{n=0}^{\infty} {\pi_n}_*(c_*^A(X^{[n]})) \cdot (-t)^n =  \prod_{k=1}^{\infty} (1-t^k \cdot d^k_*)^{-k \cdot {c}_*(X)}.
\end{equation}
\end{cor} 
In particular, comparing with formula (\ref{200}), we obtain that
\begin{equation}\label{220}
{\pi_n}_*(c_*^A(X^{[n]})) =  (-1)^n {\pi_n}_*(c_*(X^{[n]})).
\end{equation}
This formula is obtained here from a motivic viewpoint. It also follows from the constructible function identity:
\begin{equation}\label{221}
{\pi_n}_*(\nu_{X^{[n]}}) =  (-1)^n {\pi_n}_*(1_{X^{[n]}})
\end{equation}
proved in \cite{BF}[Section 4]  by localization techniques for $\C^*$-equivariant symmetric obstruction theories. 

\bibliographystyle{amsalpha}

\end{document}